\documentclass[leqno]{amsart}
\usepackage{amssymb}

\title[Pseudo-Riemannian Symmetries on Heisenberg groups]{Pseudo-Riemannian Symmetries on Heisenberg groups}

\author[Goze - Piu - Remm]{Michel Goze, Paola Piu and Elisabeth Remm}
\thanks{The first author was supported by: Visiting professor program, Regione Autonoma della Sardegna - Italy. The second author was supported by a Visiting Professor fellowship at Universit\'e de Haute Alsace - Mulhouse in February 2012 and March 2013  and by GNSAGA(Italy)}
\address{Dipartimento di Matematica e Informatica, Universit\`a degli Studi di Cagliari, Via Ospedale 72, 09124 Cagliari, ITALIA}
\email{piu@unica.it}


\address{Laboratoire de Math\'ematiques et Applications,
        Universit\'e de Haute Alsace, Facult\'e des Sciences et
        Techniques, 4, rue des Fr\`eres Lumi\`ere,
        68093~Mulhouse~cedex, France.}
\email{Michel.Goze@uha.fr, Elisabeth.Remm@uha.fr}

\begin{document}

\newtheorem{theorem}{Theorem}
\newtheorem{acknowledgement}[theorem]{Acknowledgement}
\newtheorem{algorithm}[theorem]{Algorithm}
\newtheorem{axiom}[theorem]{Axiom}
\newtheorem{case}[theorem]{Case}
\newtheorem{claim}[theorem]{Claim}
\newtheorem{conclusion}[theorem]{Conclusion}
\newtheorem{condition}[theorem]{Condition}
\newtheorem{conjecture}[theorem]{Conjecture}
\newtheorem{corollary}[theorem]{Corollary}
\newtheorem{criterion}[theorem]{Criterion}
\newtheorem{definition}[theorem]{Definition}
\newtheorem{example}[theorem]{Example}
\newtheorem{exercise}[theorem]{Exercise}
\newtheorem{lemma}[theorem]{Lemma}
\newtheorem{notation}[theorem]{Notation}
\newtheorem{problem}[theorem]{Problem}
\newtheorem{proposition}[theorem]{Proposition}
\newtheorem{remark}[theorem]{Remark}
\newtheorem{solution}[theorem]{Solution}
\newtheorem{summary}[theorem]{Summary}
\newcommand\R{\mathbb{R}}
\newcommand\HH{\mathbb{H}}
\newcommand\g{\frak{g}}
\newcommand\K{\mathbb{K}}
\newcommand\C{\mathbb{C}}
\newcommand\Z{\mathbb{Z}}
\newcommand\si{\sigma}
\newcommand\ga{\Gamma}
\newcommand{\zz}{$\Z^2_2$-symmetric space}
\newcommand\z{ \mathbb{Z}_2^2}
\newcommand\p{\mathcal{P}}
\newcommand{\pf}{\noindent {\it Proof. }}
\newcommand\f{\frak{t}}
\newcommand\h{\frak{h}}
\newcommand\m{\frak{m}}
\newcommand\ad{\mathrm{ad}}
\setlength{\oddsidemargin}{0in}
\setlength{\evensidemargin}{.25in}
\setlength{\textwidth}{6.25in}

\pagestyle{myheadings}

\begin{abstract}
The notion of  $\Gamma$-symmetric space is a natural generalization of the classical notion of  symmetric space based on $\Z_2$-grading of Lie algebras. In our case, we consider homogeneous spaces $G/H$ such that the Lie algebra $\g$ of $G$ admits a $\Gamma$-grading where $\Gamma$ is a finite abelian group.
In this work we study Riemannian metrics and Lorentzian metrics on the Heisenberg group $\mathbb{H}_3$ adapted to the symmetries of a $\Gamma$-symmetric structure on $\mathbb{H}_3$. We prove that the classification of $\z$-symmetric Riemannian and Lorentzian metrics on $\mathbb{H}_3$ corresponds to the classification  of left-invariant Riemannian and Lorentzian metrics, up to isometry. We study also the $\Z_2^k$-symmetric structures on $G/H$ when $G$ is the $(2p+1)$-dimensional Heisenberg group for $k \geq 1$. This gives examples of non riemannian symmetric spaces. When $k \geq 1$, we show that there exists a family of flat and torsion free affine connections adapted to the $\Z_2^k$-symmetric structures.
\end{abstract}

\maketitle

\bibliographystyle{plain}

\section{Introduction}

A symmetric space can be considered as a reductive homogeneous space $G/H$ on which acts an abelian subgroup $\Gamma$ of the automorphisms group of $G$ with $\Gamma$ isomorphic to $\Z_2=\Z/ 2\Z$ and $H$ the subgroup of $G$ composed of the fixed points of the automorphisms belonging to $\Gamma.$ If we suppose that the Lie groups $G$ and $H$ are connected and that $G$ is simply connected, it is equivalent to provide $G/H$ with a symmetric  structure or to provide the Lie algebra $\mathfrak{g}$ of $G$ with a $\Z_2$-graduation $\g =\g_0 \oplus \g_1$ with $[\g_i,\g_j]=\g_{i+j(mod \ 2)}.$  Riemannian symmetric spaces form an interesting class of symmetric spaces. But there are symmetric spaces which are not Riemannian symmetric. We describe examples when $G$ is the Heisenberg group. Nevertheless, a symmetric space is always provided with an affine connection $\nabla$ which is torsion free and has a curvature tensor satisfying $\nabla R=0$. When the symmetric space is Riemannian, this connection is the Levi-Civita connection of the metric. A natural generalization of the notion of symmetric space can be obtained by considering that the subgroup $\Gamma$ is abelian, finite and not necessarily isomorphic to $\Z_2.$ When $\Gamma$ is cyclic isomorphic to $\Z_k$ it corresponds to the generalized symmetric spaces of \cite{[K],L.O}. These structures are also characterized  by $\Z_k$-graduations of the complexified Lie algebra $\g_\C=\g \oplus \C$ of $\g.$ We get another interesting case when $\Gamma=\Z_2^k$ because the characteristic graduation is defined on $\g$ and not on $\g_\C.$ When $\g$ is simple the $\Z_2^2$-graduations of $\g$ have been classified as well as the $\Z_2^2$-symmetric spaces $G/H$ when $G$ is simple connected (\cite{[B.G], [Ko]}). All these spaces are Riemannian (see \cite{[P.R]}). But, in this paper, we provide some examples of non Riemannian symmetric spaces studying symmetric spaces  $G/H$ when $G$ is the Heisenberg group $\mathbb{H}_{2p+1}.$ We study also $\Z^k_2$-symmetric structures on these homogeneous spaces showing, in particular, that these spaces are Riemannian and affine. But contrary to the symmetric case, there exist on these spaces affine connections different from the canonical (or the Levi-Civita) connection and more adapted to the symmetries of $G/H$ that the canonical one. We describe  these connections and we prove that there exists connections adapted to the $\Z_2^k$ symmetries which are flat and torsion free.

\section{$\Z_2^k $-symmetric spaces}
\subsection{Recall on symmetric and Riemannian symmetric spaces}
A symmetric space is a triple $(G,H,\sigma)$ where $G$ is a connected Lie group, $H$ a closed subgroup of $G$ and $\si$ an involute automorphism of $G$ such that 
$G_e^\si \subset H \subset G^\si$ where $G^\si=\{x \in G, \ \si (x)=x\}$, $G_e^\si$ the identity component of $G^\si$. If $(G,H,\si)$ is a symmetric space, at each point $\overline{x}$ of the homogeneous manifold $M=G/H$ corresponds  an involutive diffeomorphism $\si_{\overline{x}}$ which has $\overline{x}$ as an isolated fixed point. Let $\g$ and $\h$ be the Lie algebras of $G$ and $H$. The automorphism $\si \in Aut(G)$ induces an involutive automorphism of $\g$, denoted by $\si$ again,  such that $\h$ consists of all elements of $\g$ which are left fixed by $\si$. We deduce that the Lie algebra $\g$ is $\Z_2$-graded:
$$\g=\h \oplus \m$$
with $\m=\{X \in \g, \ \si (X)=-X\}$ and $[\h,\m] \subset \m$, $[\m,\m] \subset \h$ and $[\h,\h] \subset \h$. If we assume that $G$ is simply connected and $H$ connected, then the $\Z_2$-grading $\g=\h \oplus \m$ defines a  symmetric space stucture $(G,H,\si)$. Thus, under these hypothesis, it is equivalent to speak about $\Z_2$-grading of Lie algebras or symmetric spaces.

An important class of symmetric spaces consists of Riemannian symmetric spaces. A Riemannian symmetric space is a Riemannian manifold $M$ whose curvature tensor field associated with the Levi-Civita connection is parallel. In this case the geodesic symmetry at a point $u \in M$ attached to the Levi-Civita connection is an isometry and, if we fix $u$, it defines an involutive automorphism $\si$ of the largest group of isometries $G$ of $M$ which acts transitively on $M$. We deduce that $M$ is an homogeneous manifold $M=G/H$ and the triple $(G,H,\si)$ is a symmetric space.  Let us note that, in this case, $H$ is compact. When $H \cap Z(G) =\{e\},$ this last condition is equivalent to $ad_{\g}(H)$ compact. Here $Z(G)$ denotes the center of $G$. Conversely, if $(G,H,\si)$ is a symmetric  space such that the image $ad_\g (H)$ of $H$ under the adjoint representation of $G$ is a compact subgroup of $Gl(\g)$, then $\g$ admits an $ad_{\g}(H)$-invariant inner product and  $\h$ and $\m$ are orthogonal with respect to it. This inner product restricted to $\m$ induces an $G$-invariant Riemannian metric on $G/H$ and $G/H$ is a Riemannian symmetric space.  For example, if $H$ is compact, $ad_{\g}(H)$ is also compact  and $(G,H,\si)$ is a Riemannian symmetric space. Assume now that $H$ is connected, then $ad_{\g}(H)$ is compact if and only if the connected Lie group associated with the linear algebra $ad_\g(\h)=\{ad X, \ X \in \h\}$ is compact. In this case, $\g$ admits an $ad_{\g}(\h)$-invariant inner product $\varphi,$ that is,
$$\varphi([X,Y],Z)+\varphi(Y,[X,Z])=0$$
for all $X \in \h$ and $Y,Z \in \g$ such that $\varphi(\h,\m)=0.$ An interesting particular case is the following. Assume that $\g$ is $\Z_2$-graded and that this grading is effective that is $\h$ doesn't contain non trivial ideal of $\g$. If $ad_{\g}(\h)$ is irreducible on $\m$, then $\g$ is simple, or a sum $\g_1 + \g_1$ with $\g_1$ simple or $\m$ abelian. In the first case, the Killing-Cartan form $K$ of $\g$ induces a negative or positive defined bilinear form on $\m$. It follows a classification of $\Z_2$-graded Lie algebras when $\g$ is simple or semi-simple.

Many results on the problem of classifications concern more particularly  the simple Lie algebras. For solvable or nilpotent Lie algebras, it is an open problem. A first approach is to study induced grading on Borel or parabolic subalgebras of simple Lie algebras. In this work we describe $\Gamma$-grading of the Heisenberg algebras. Two reasons for this study

$\bullet$ Heisenberg algebras are nilradical of some Borel subalgebras.

$\bullet$ The Riemannian and Lorentzian geometries on the $3$-dimensional Heisenberg group have been studied recently by many authors. Thus it is interesting to study the Riemannian and Lorentzian symmetries with the natural symmetries associated with a $\Gamma$-symmetric structure on the Heisenberg group. In this paper we prove that these geometries are entirely determinated by Riemannian and Lorentzian structures adapted to $\Z_2^2$-symmetric structures.
\subsection{$\Gamma$-symmetric spaces}
Let $\Gamma$ be a finite abelian group.
\begin{definition}
A $\Gamma$-symmetric space is a triple $(G,H,\tilde{\Gamma})$ where $G$ is a connected Lie group, $H$ a closed subgroup of $G$ and $\tilde{\Gamma}$ a finite abelian subgroup of the group ${\rm Aut}(G)$ of automorphisms of $G$ isomorphic to $\Gamma$ such that
$G_e^\Gamma \subset H \subset G^\Gamma$ where $G^\Gamma=\{x \in G, \ \si (x)=x\ \ \forall \si \in \tilde{\Gamma}\}$, $G_e^\Gamma$ the identity component of $G^\Gamma$.
\end{definition}
\noindent If $\Gamma$ is isomorphic to $\Z_2$ then we find the notion of symmetric spaces again.  If $\Gamma$ is isomorphic to $\Z_k$ with $k \geq 3$, then $\Gamma$ is a cyclic group generated by an automorphism of order $k$. The corresponding spaces are called generalized symmetric spaces and have been studied by A.J. Ledger, M. Obata \cite{L.O}, A. Gray, J. A. Wolf,  \cite{[G.W]} and O. Kowalski \cite{[K]}. The general notion of $\Gamma$-symmetric spaces was introduced by R. Lutz \cite{[L]} and was algebraically reconsidered by Y. Bahturin and M. Goze \cite{[B.G]}.

\noindent An equivalent and useful definition is the following:

\begin{definition}
Let $\Gamma $ be a finite abelian group. A $\Gamma$-{\it symmetric space} is an homogeneous space $ G/H $ such that there exists an injective homomorphism
\[
\rho : \Gamma \to Aut(G)
\]
 where $ Aut (G) $ is the group of automorphisms of the Lie group $ G $, the subgroup $ H $ satisfies $G^{\Gamma}_{e} \subset H \subset G^{\Gamma}$  where
$ G^{\Gamma} = \left\{ x\in G/ \rho(\gamma)(x) = x, \forall \gamma \in \Gamma \right\}$ and $G^{\Gamma}_{e}$ is the connected identity component of $ G^{\Gamma} $  of $G$.
\end{definition}

In \cite{[B.G]}, one proves that, if $G$ and $H$ are connected, then the triple $(G,H,\tilde{\Gamma})$ is a $\Gamma$-symmetric space if and only if the complexified Lie algebra $\g_\C=\g \otimes_\R \C$ of $\g$ is $\Gamma$-graded:
$$\g_\C= \bigoplus _{\gamma \in \Gamma}\g_\gamma$$
where $\g_\epsilon=\h$ is the Lie algebra of $H$ with $\epsilon$ is the unit of $\Gamma$. In this case, we have  the relations
\[
[\mathfrak{g}_{\gamma},\mathfrak{g}_{\gamma'}] \subset \mathfrak{g}_{\gamma \gamma'}
\qquad \qquad  \forall \; \gamma, \gamma' \in \Gamma.
\]
In fact, the derivative of an automorphism $\sigma$ of $G$ belonging to $\tilde{\Gamma}$ is an automorphism of $\g$, still denoted $\sigma.$ So if $\gamma$ runs over $\tilde{\Gamma}$, we obtain a subgroup $\hat{\Gamma}$ of the group of automorphisms of $\g$ which is isomorphic to $\Gamma$.
The elements of $\hat{\Gamma}$ are automorphisms of $\g$ of finite order, pairwise commuting and the $\Gamma$-grading corresponds to the spectral decomposition of $\g_\C$ associated with the abelian finite group $\hat{\Gamma}$. Conversely, if we have a $\Gamma$-grading of $\g_\C$, and if we denote by $\check{\Gamma}$ the dual group of $\Gamma$, that is, the group of characters, thus $\check{\Gamma}$ is a finite abelian group isomorphic to $\Gamma$. Any element $\chi \in \check{\Gamma}$ can be considered as an automorphism of $\g_\C$ by
$$\chi(X)=\chi(\gamma)X$$
for any homogeneous vector $X \in \g_\gamma$. Thus $\check{\Gamma}$ is an abelian subgroup of $Aut(\g_\C)$ isomorphic to $\Gamma$ and the $\Gamma$-grading of $\g$ corresponds to the spectral decomposition associated with 
$\check{\Gamma}$ considered as an abelian finite subgroup of  $Aut(\g_\C)$. Then, if we assume that $G$ is also simply connected, we have a one-to-one correspondence    between the set of  $\Gamma$-symmetric stuctures and the $\Gamma$-gradings of $\g$.

 In \cite{[L]}, it is shown that for any $\bar{x }\in M=G/H$, there exists a subgroup $\Gamma_{\bar{x}}$ of the group ${\mathcal{D}}{\rm iff}(M)$ of diffeomorphisms of $M$, isomorphic to $\Gamma$, such that $\bar{x}$ is the unique point of $M$ satisfying $\sigma(\bar{x})=\bar{x}$ for any $\sigma \in \Gamma_{\bar{x}}.$ By extension, the elements of $\Gamma_{\bar{x}}$ are also called symmetries of $M$.

\subsection{$\Z_2^k$-symmetric spaces}

Assume that $\Gamma=\Z_2^k$. In this case any element of $\hat{\Gamma}$ is an involutive automorphism of $\g$  and the eigenvalues are real. Since the elements of $\hat{\Gamma}$ are pairwise commuting, we define a spectral decomposition of $\g$ itself. This implies a $\Z_2^k$-grading  defined on $\g$  :
  \[
\mathfrak{g}=  \underset{\gamma \in \Gamma}{\bigoplus}\; \mathfrak{g}_{\gamma}.
\]
For example, if $k=2$, then $\Gamma=\{a,b,c,\epsilon\}$ where $\epsilon$ is the identity, with  
$$a^2=b^2=c^2=\epsilon, \quad ab=c, \quad bc=a, \quad ca= b.$$ and $\hat{\Gamma}$ contains $4$ elements, $\si _a,\si _b, \si_c$ and the identity $Id$. These maps are involutive and satisfy
$$\si _a \circ \si _b=\si _c, \quad \si _b \circ \si _c=\si _a,  \quad \si _c \circ \si _a=\si _b .$$ 
Each one of these linear maps is diagonalizable, and because they are pairwise commuting, we can diagonalize all these maps simultaneously. Let $\g_a=\{X \in \g, \si_a(X)=X,\si_b(X)=-X\}$,  $\g_b=\{X \in \g, \si_a(X)=-X,\si_b(X)=X\}$,  $\g_c=\{X \in \g, \si_a(X)=-X,\si_b(X)=-X\}$  and $\g_\epsilon=\{X \in \g, \si_a(X)=X,\si_b(X)=X\}$ be the root spaces. We have 
$$\g=\g_\epsilon \oplus \g_a \oplus \g_b \oplus \g_c.$$

Let us return to the general case $\Gamma=\Z_2^k$. If $G$ is connected and simply connected and $H$ connected, then the $\Gamma$-grading of $\g$ determine a structure of $\Gamma$-symmetric space on the triple $(G,H,\tilde{\Gamma})$. We will say also that the homogeneous space $G/H$ is a $\Z_2^k$-symmetric space. 
\medskip

\begin{proposition}
Any  $\Z_2^k$-symmetric space homogeneous space $G/H$ is reductive.
\end{proposition}
\pf In fact if $\g = \underset{\gamma \in \Z_2^k}{\bigoplus}\; \mathfrak{g}_{\gamma}$ is the associated decomposition of $\g$, thus putting 
$$\m=~\underset{\gamma \in \Gamma, \gamma \neq \epsilon}{\bigoplus}\; \mathfrak{g}_{\gamma},$$
we have $\g= \g_\epsilon \oplus \m$ with $[\g_\epsilon,\g_\epsilon] \subset \g_\epsilon$ and $[\g_\epsilon,\m ] \subset \m$. The decomposition $\g=\g_\epsilon \oplus \m$ is reductive. In general $[\m,\m] $  is not a subset  of $\g_\epsilon$, except if $k=1$.

\medskip

Two $ \Z_2^k$-gradings  $ \g=\bigoplus_{\gamma \in \Z_2^k}\g_\gamma \  {\rm and} \ \g=\bigoplus_{\gamma' \in \Z_2^k}\g'{_{\gamma'}} \ {\rm of} \ \g $ are called equivalent  if  there   exist an automorphism $\pi$ of $\g$ and an automorphism $\omega$ of $\Z_2^k$ such that
$$\g'{_{\gamma'}}=\pi(\g_{\omega(\gamma)}) \qquad {\rm for \  any} \qquad \gamma' \in \Z_2^k.$$
If we consider only connected and simply connected groups $G$, and connected subgroups $H$, then the classification of $\Z_2^k$-symmetric spaces is equivalent to the classification, up to equivalence, to $\Z_2^k$-gradings on Lie algebras. For example,
the $\Z_2 ^2$-grading of classical simple complex Lie algebras are classified in \cite{[B.G]}. This classification is completed for exceptional simple algebras in \cite{[Ko]}.

\subsection{Riemannian and pseudo-Riemannian $\Z_2^k$-symmetric spaces}
Let $(G,H,\Z^k_2)$ be a $\Z_2^k$-symmetric space with $G$ and $H$ connected. The homogeneous space $M=G/H$ is reductive. Then there exists a one-to-one correspondence between the $G$-invariant pseudo-Riemannian metrics $g$ on $M$ and the non degenerated symmetric bilinear form $B$ on $\m$ satisfying
$$B([Z,X],Y)+B(X,[Z,Y])=0$$
for all $X,Y \in \m$ and $Z \in \g_\epsilon$. 

\begin{definition}\cite{[G.R]} \
A $\Z_2^k$-symmetric space $M=G/H$ with $Ad_G(H)$  compact, is called Riemannian  $\Z_2^k$-symmetric if $M$ is provided with a $G$-invariant Riemannian metric $g$ whose associated bilinear form $B$ satisfies
\begin{enumerate}
\item $B(\mathfrak g_{\gamma},\mathfrak g_{\gamma'}) = 0$ if  $\gamma \neq \gamma' \neq \epsilon,$
\item The restriction of $ B $ to $\m=\oplus_ {\gamma\neq \epsilon}\mathfrak g_\gamma $ is  positive definite.
\end{enumerate}
\end{definition}
In this case the linear automorphisms which belong to $\hat{\Gamma}$ are linear isometries. Some examples are described in \cite{[P.R]}.

\begin{proposition}
Let $(G,H,\Z_2^k)$ be a Riemannian $\Z_2^k$-symmetric space, $G$ and $H$ supposed to be connected. Then $H$ is compact.
\end{proposition}
\pf In fact, $H$ coincides with the identity component of the isotropy group which is compact.

\medskip

\noindent{\bf Example: $\Z_2^k$-symmetric nilpotent spaces.} Let $(G,H,\Z_2^k)$ be a  $\Z_2^k$-symmetric space with $G$ nilpotent. Such a space will be called a $\Z_2^k$-symmetric nilpotent spaces. If $k=1$, we cannot have on $G/H$ a symmetric Riemannian metric except if $G$ is abelian. But, if $k \geq 2$, there exist $\Z_2^k$-symmetric Riemannian nilpotent spaces. For example, let $G$ be the $3$-dimensional Heisenberg Lie group. Its Lie algebra $\h_3$ admits a basis $\{X_1,X_2,X_3\}$ with $[X_1,X_2]=X_3$. We have a $\Z_2^2$-grading of $\h_3$:
$$\h_3=\{0\} \oplus \R\{X_1\}\oplus \R\{X_2\}\oplus \R\{X_3\}$$
and the metric 
$$g=\omega_1^2+\omega_2^2+\omega_3^2$$
defines a structure of $\Z_2^k$-symmetric Riemannian nilpotent spaces on $H_3/ \{e\}=H_3$ where 
$\{\omega_1,\omega_2,\omega_3\}$ is the dual basis of $\{X_1,X_2,X_3\}$. We will develop this calculus in the next sections.

\medskip

A Lorentzian metric on a $n$-dimensional differential manifold $M$ is a smooth field  of non degenerate quadratic forms of signature $(n-1,1)$.
We say that a homogeneous space $(M = G/H, g)$ provided with a Lorentzian metric $g$ is {\it Lorentzian} if the canonical action of $G$ on $M$ preserves the metric. If $M$ is reductive and if $\g=\g_\epsilon \oplus \m$, the Lorentzian metric is determinate by the $ad \g_\epsilon$-invariant non degenerate bilinear form $B$ with signature $(n-1,1)$.

\begin{definition}
Let $(G,H,\Z^k_2)$ be a $\Z_2^k$-symmetric space. It is called Lorentzian if there exists on the homogeneous space $ M = G / H$  a Lorentzian metric $g$  such that one of the two conditions is satisfied:
\begin{enumerate}
\item The homogeneous non trivial components $\g_\gamma$ of the $\Z_2^k$-graded Lie algebra $\g$ are orthogonal and non-degenerate with respect to the induced bilinear form $ B $.
\item One non trivial component $\g_{\lambda_0}$ is degenerate, the other components are orthogonal and non-degenerate,  and the exists a component $\g_{\lambda_1}$ such that the signature of the restriction to $B$ at $\g_{\lambda_0} \oplus \g_{\lambda_1}$ is  $(1,1)$.
\end{enumerate}
\end{definition}
Let us note that, in this case, $H$ is not necessarily compact. Some  examples of $\Z_2^k$-symmetric nilpotent Lorentzian spaces are described in the next sections. 

\section{Affine structures on $\Z_2^k$-symmetric spaces}

Let $(G,H,\Z_2^k)$ be a $\Z_2^k$-symmetric space. Since the homogeneous space $G/H$ is reductive, 
from \cite{Kobayashi-Nomizu-2}, Chapter X, we deduce that $M=G/H$ admits  two $G$-invariant canonical connections 
denoted by $\nabla$ and $\overline \nabla$. The \emph{first canonical connection}, $\nabla$, satisfies
$$
\left\{
\begin{array}{l}
R(X,Y)=-\ad([X,Y]_\h), \qquad T(X,Y)_{\overline e}=-[X,Y]_{\mathfrak{m}}, \quad \forall X,Y \in \mathfrak{m} \\
\nabla T=0 \\
\nabla R=0
\end{array}
\right.
$$
where $T$ and $R$ are the torsion and the curvature tensors of $\nabla$. The tensor $T$ is trivial
if and only if $[X,Y]_{\mathfrak{m}}=0$ for all $X,Y \in \mathfrak{m}$. This means that $[X,Y] \in \h$ that is $[\mathfrak{m},\mathfrak{m}]\subset \h$.
 If the grading of $\g$ is given by $\Z_2^k$ with $k >1$, then $[\mathfrak{m},\mathfrak{m}]$ is not  a subset of  $\h$ and then the torsion $T$ need not to vanish.
In this case the another connection $\overline \nabla$ is given by $\overline \nabla _X Y= \nabla_XY-T(X,Y)$. This is an affine invariant torsion free connection
on $G/H$ which has the same geodesics as $\nabla$. This connection is called the \emph{second canonical connection} or the \emph{torsion-free canonical
connection}. 
\medskip

\noindent{\bf Remark.} Actually, there is another way of writing the canonical 
affine connection of a $\Gamma$-symmetric space, without any reference to Lie algebras. 
This is done by an intrinsic construction of $\Gamma$-symmetric spaces  proposed by Lutz 
in \cite{[L]}.

\subsection{Associated affine connection }
 Any symmetric space $G/H$ is an affine space, that is, it is provided with an affine connection $\nabla$ whose torsion tensor $T$ and curvature tensor $R$ satisfy
 $$T=0, \qquad \nabla R=0$$
 where
 $$
 \begin{array}{ll}
 \nabla R(X_1,X_2,X_3,Y)=&\nabla(Y,R(X_1,X_2,X_3))-R(\nabla(Y,X_1),X_2,X_3)\\
& -R(X_1,\nabla(Y,X_2),X_3)-R(X_1,X_2,\nabla(Y,X_3))
 \end{array}
 $$
 for any vector fields $X_1,X_2,X_3,Y$ on $G/H$. It is the only affine connection which is invariant by the symmetries of $G/H$. This means that the two canonical connections, which are defined on an homogeneous reductive space, coincides if the reductive space is symmetric. For example, if $G/H$ is a Riemannian symmetric space, this connection $\nabla$ coincides with the Levi-Civita connection associated with the Riemannian metric.

 \medskip
 
 Let us return to the general case. Let us assume that  $G/H$ is a reductive homogeneous space, and let $\g=\h \oplus \m$ be the reductive decomposition of $\g$.  Any connection on $G/H$ is given by a linear map
 $$\begin{array}{l}
\bigwedge : \m \rightarrow gl(\m)
\end{array}
$$
 satisfying
 $$
\begin{array}{l}
\bigwedge[X,Y]=[\bigwedge (X), \lambda (Y)]
\end{array}
$$
 for all $X \in \m$ and $Y \in \h$, where $\lambda$ is the linear isotropy representation of $\h$. The corresponding torsion and curvature tensors are given by:
 $$\begin{array}{l}
T(X,Y)=\bigwedge(X)(Y) - \bigwedge(Y)(X)-[X,Y]_{\m}
\end{array}$$
 and
 $$\begin{array}{l}
R(X,Y)=[\bigwedge(X),\bigwedge(Y)]-\bigwedge[X,Y]-\lambda([X,Y]_{\h})
\end{array}
$$
 for any $X,Y \in \m.$

 \medskip
 
 Let $(G,H,\Z_2^k)$ be a $\Z_2^k$-symmetric space. We have recalled that, when $k=1$, the homogeneous space $G/H$ is an affine space. But, as soon as $k>1$, in general the two canonical connections do not coincide and the torsion tensor of the first one is not trivial.  We   can consider connections adapted to the $\Z_2^k$-symmetric structures.
 
\begin{definition} \label{adapted}
Let $\nabla$ an affine connection on the $\Z_2^k$-symmetric space $G/H$ defined by the linear map
$$\bigwedge : \m \rightarrow gl(\m).$$
Then this connection is called adapted to the $\Z_2^k$-symmetric structure, if 
$$\bigwedge(X_\gamma)(\g_{\gamma'}) \subset \g_{\gamma\gamma'}$$
for any $\gamma,\gamma' \in \Z_2^k$, $\gamma,\gamma' \neq \epsilon.$ The connection is called homogeneous if any
homogeneous component $\g_{\gamma}$ of $\m$ is invariant by $\bigwedge$.
\end{definition}
\noindent{\bf Examples}
\begin{enumerate}
\item If $k=1$, the affine canonical connection is adapted and homogeneous.
\item Let us consider the $5$-dimensional nilpotent Lie algebra, $\mathfrak{l}_5$ whose Lie brackets are given in a basis $\{X_1,\cdots,X_5\}$ by 
$$[X_1,X_i]=X_{i+1},\qquad i=2,3,4.$$
This algebra admits a $\Z_2$-grading
 $$\mathfrak{l}_5=\R\{X_3,X_5\} \oplus \R\{X_1,X_2,X_4\}.$$
 Thus $\bigwedge(X_1),\bigwedge(X_2),\bigwedge(X_4)$ are matrices of order $3$. If we assume that the torsion $T$ is zero, we obtain
 $$\begin{array}{l}
 \bigwedge(X_1)=\left(
 \begin{array}{lll}
 a & 0 & 0\\
 b & 0 & 0 \\
 c & d & \frac{a}{2}
 \end{array}
 \right)
 , \ \
 \bigwedge(X_2)=\left(
 \begin{array}{lll}
 0 & 0 & 0\\
0 & e & 0 \\
 d & f & \frac{a}{2}
 \end{array}
 \right)
, \ \
 \bigwedge(X_3)=\left(
 \begin{array}{lll}
 0 & 0 & 0\\
0 & 0 & 0 \\
 -\frac{a}{2} & 0 & 0
 \end{array}
 \right).
\end{array} 
$$
 The linear isotropy representation of $H$ whose Lie algebra is $\h$ is given by taking the differential of the map $\mathfrak{l}_5/H\rightarrow \mathfrak{l}_5/H$ corresponding to the left multiplication $ \overline{x}\rightarrow h\overline{x}$ with $\overline{x}=xH_4$. We obtain
 $$
 \lambda(X_3)=\left(
 \begin{array}{lll}
0 & 0 & 0\\
0 & 0 & 0 \\
 1 & 0 & 0
 \end{array}
 \right)
 , \qquad
\lambda(X_5)=\left(0
 \right).
 $$
 We deduce that the curvature is always non zero.
\end{enumerate}


\section{The $\Z_2^k$-symmetric spaces $(\HH_3,H,\Z_2^k)$}
 We denote by $\HH_3$ the $3$-dimensional Heisenberg group, that is the linear group of dimension $ 3 $ consisting of matrices
\[
\begin{pmatrix}
1&a&c\\
0&1&b\\
0&0&1
\end{pmatrix}
\qquad \qquad a,b,c \in \R. 
\]
Its Lie algebra, $\h_3$ is the real Lie algebra whose elements are matrices
\[
\begin{pmatrix}
0&x&z\\
0&0&y\\
0&0&0
\end{pmatrix}
\qquad  \mbox{with} \qquad x,y,z \in \R.
\]
The elements of  $\mathfrak h_{3}$, $X_1$, $X_2$, $X_3$, corresponding to $(x,y,z) =(1,0,0),(0,1,0)$ and $(0,0,1)$ form a basis of $\mathfrak h_{3}$ and the Lie brackets are given in this basis by
$$[X_1,X_2] = X_3, \quad \quad  [X_1,X_3] = [X_2,X_3] = 0 .$$

\subsection{Description of $Aut(\h_3)$}
Denote by $Aut(\h_3)$ the group of automorphisms  $\h_3$. Every  $\tau \in Aut(\h_3)$ admits in the basis $\{X_1,X_2,X_3\}$ the following matricial representation:
\begin{equation} \label{aut}
\begin{pmatrix}
\alpha_1&\alpha_2&0\\
\alpha_3&\alpha_4&0\\
\alpha_5&\alpha_6&\Delta
\end{pmatrix} \quad {\rm with} \quad \Delta = \alpha_1 \alpha_4 - \alpha_2 \alpha_3 \neq 0.
\end{equation}
We will denote by $\tau(\alpha_1,\alpha_2,\alpha_3,\alpha_4,\alpha_5,\alpha_6)$ any element of $Aut(\h_3)$ with this representation.
Let $\Gamma $ be a finite  abelian subgroup of $Aut(\h_3)$. It admits a cyclic decomposition. If $\Gamma $ contains a component of the cyclic decomposition which is isomorphic to $\Z_k $,  then there exists an automorphism $\tau $ satisfying $\tau^ k = Id $. The aim of this section is to determinate  the cyclic decomposition of any finite abelian subgroup $\Gamma$.

\medskip

\noindent $\bullet$ {\bf Subgroups of $Aut(\h_3)$ isomorphic to $\Z_2$}

Let $\tau \in Aut(\h_3)$ satisfying $\tau^2 = Id $. If we consider the matricial representation (\ref{aut}) of $\tau$, we obtain:
\[
\begin{pmatrix}
\alpha_1^{2}+ \alpha_{2}\alpha_{3}&\alpha_{1}\alpha_2 + \alpha_{2}\alpha_{4}&0\\
\alpha_1\alpha_3+\alpha_3 \alpha_4&\alpha_{2}\alpha_{3}+\alpha_4^{2}&0\\
\alpha_1\alpha_5+ \alpha_3 \alpha_6 + \Delta \alpha_5&
\alpha_2\alpha_5+\alpha_4\alpha_6+\Delta \alpha_6&\Delta^{2}
\end{pmatrix} =
\begin{pmatrix}
1&0&0\\
0&1&0\\
0&0&1
\end{pmatrix}.
\]
\begin{proposition}\label{lemma1}
Any involutive automorphism $\tau $ of $ Aut (\mathfrak h_3) $ is equal to one of the following automorphisms
$$\begin{array}{c}
Id, \qquad \tau_1(\alpha_3,\alpha_6) = \begin{pmatrix}
-1&0&0\\
\alpha_3&1&0\\
\displaystyle\frac{\alpha_{3}\alpha_6}{2}&\alpha_6&-1
\end{pmatrix},  \qquad  \tau_2(\alpha_3,\alpha_5) = \begin{pmatrix}
1&0&0\\
\alpha_3&-1&0\\
\alpha_5&0&-1
\end{pmatrix},
\end{array}$$
$$\begin{array}{c}
\tau_3(\alpha_1,\alpha_{2}\neq 0,\alpha_6) = \begin{pmatrix}
\alpha_1&\alpha_{2}&0\\
\displaystyle\frac{1-\alpha_1^{2}}{\alpha_{2}}&-\alpha_{1}&0\\
\displaystyle\frac{(1 + \alpha_{1})\alpha_6}{\alpha_{2}}&\alpha_6&-1
\end{pmatrix},  \qquad
 \tau_4(\alpha_5,\alpha_6) = \begin{pmatrix}
-1&0&0\\
0&-1&0\\
\alpha_5&\alpha_{6}&1
\end{pmatrix}  .
\end{array}$$
\end{proposition}
\begin{corollary}
Any subgroup of $Aut(\h_3)$ isomorphic to $\Z_{2}$ is  one of the following:
$$
\begin{array}{ll}
\Gamma_1(\alpha_3,\alpha_6)=\{Id,\tau_1(\alpha_3,\alpha_6)\}, &
\Gamma_2(\alpha_3,\alpha_5)=\{Id,\tau_2(\alpha_3,\alpha_5)\}, \\
\Gamma_3(\alpha_1,\alpha_2,\alpha_6)=\{Id,\tau_3(\alpha_1,\alpha_2,\alpha_6), \ \alpha_2 \neq 0\}, &
\Gamma_4(\alpha_5,\alpha_6)=\{Id,\tau_4(\alpha_5,\alpha_6)\}.
\end{array}
$$
\end{corollary}

\medskip

\noindent $\bullet$ {\bf Subgroups of $Aut(\h_3)$ isomorphic to $\Z_k, \ k \geq 3$}.
If $\tau=\tau (\alpha_1,\alpha_2,\alpha_3,\alpha_4, \alpha_5,\alpha_6)\in Aut(\h_3)$ satisfies $\tau^k=Id$, then $\Delta=\alpha_1\alpha_4-\alpha_2\alpha_3=1$ and its minimal polynomial has $3
$ simple roots and it is of degree $3$. More precisely, it is written
\[
m_\tau(x)=(x-1)(x-\mu_k)(x-\overline{\mu_k})
\]
where $\mu_k$ is a root of order $k$ of $1$. Since we can assume that $\tau$  is a generator of a cyclic subgroup of $Aut(\h_3)$ isomorphic to $\Z_k$, the root $\mu_k$ is a primitive root of $1$. There exists  $m$ relatively prime with $k$ such that $\mu_k= exp\left( \displaystyle\frac{2mi\pi}{k}\right) .$ We have $\alpha_1+\alpha_4=\mu_k+\overline{\mu_k}$
and $\alpha_1+\alpha_4=2\displaystyle\cos\frac{2m\pi}{k}$. Thus
$$
\alpha_1=\displaystyle\cos\frac{2m\pi}{k} - \sqrt{\cos^2\frac{2m\pi}{k} - 1- \alpha_{2} \alpha_{3}},\qquad
\alpha_4=\displaystyle\cos\frac{2m\pi}{k} + \sqrt{\cos^2\frac{2m\pi}{k} - 1- \alpha_{2} \alpha_{3}}$$
or
$$
\alpha_1=\displaystyle\cos\frac{2m\pi}{k} + \sqrt{\cos^2\frac{2m\pi}{k} - 1- \alpha_{2} \alpha_{3}},\qquad \alpha_4=\displaystyle\cos\frac{2m\pi}{k} - \sqrt{\cos^2\frac{2m\pi}{k} - 1- \alpha_{2} \alpha_{3}}.$$
If $\tau'$ and $\tau ''$ denote the automorphisms corresponding to these solutions, we have, for a good choice of the parameters $\alpha_{i}$,  $\tau' \circ \tau '' =Id$ and $\tau''=(\tau')^{k-1}.$
Thus these automorphisms generate the same subgroup of $Aut(\h_3)$. Moreover, with same considerations, we can choose
$m=1$. Thus we have determinate the automorphism
$\tau_5(\alpha_2,\alpha_3,\alpha_5,\alpha_6)$ whose matrix is
\[
 \begin{pmatrix}
\displaystyle\cos\frac{2\pi}{k} + \sqrt{\cos^2\frac{2\pi}{k} - 1- \alpha_{2} \alpha_{3}}&\alpha_2&0\\
\alpha_3&\displaystyle\cos\frac{2\pi}{k} - \sqrt{\cos^2\frac{2\pi}{k} - 1- \alpha_{2} \alpha_{3}}&0\\
\alpha_5&\alpha_6&1
\end{pmatrix}
\]
\begin{proposition}
Any abelian subgroup of $Aut(\mathfrak h_3) $ isomorphic to $\Z_k$,$k \geq 3$, is equal to
\[
\Gamma_{6,k}(\alpha_2,\alpha_3,\alpha_5,\alpha_6) =\left\{Id,\tau_{6}(\alpha_2, \alpha_3, \alpha_5, \alpha_6),\cdots,\tau_{6}^{k-1}, \ \ \alpha_{2} \alpha_{3} \leq -1+ \displaystyle\cos^2\frac{2\pi}{k}\right\}.
\]
\end{proposition}

\medskip

\noindent{\bf General case.}
Suppose now that the cyclic decomposition of a finite abelian subgroup $\Gamma$ of $Aut(\h_3)$ is isomorphic to
$\Z_2^{k_2} \times \Z_3^{k_3} \times \cdots \times \Z_p^{k_p}$ with $k_i \geq 0.$
\begin{lemma}
Let $\Gamma$ be an abelian finite subgroup of $Aut(\h_3)$ with a cyclic decomposition isomorphic to
\[
\Z_2^{k_2} \times \Z_3^{k_3} \times \cdots \times \Z_p^{k_p}.
\]
Then
 \begin{enumerate}
 \item If there is $i\geq 3$ such that $k_i \neq 0$, then $k_2 \leq 1.$
 \item If $k_2 \geq 2$, then $\Gamma$ is isomorphic to $\Z_2^{k_2}$.
\end{enumerate}
\end{lemma}
\begin{proof}
Assume that there is $i \geq 3$ such that $k_i \geq 1$.  If $k_2 \geq 1$, there exist two automorphisms $\tau$ and $\tau'$ satisfying $\tau'^{i}=\tau^{2}=Id$ and $\tau' \circ \tau=\tau\circ \tau'.$ Thus $\tau'$ and $\tau$ can be reduced simultaneously in the diagonal form and admit a common basis of eigenvectors. Since for any $\sigma\in Aut(\h_3)$ we have $\sigma(X_3)=\Delta X_3$, $X_3$ is an eigenvector for $\tau'$ and $\tau$ associated to the eigenvalue $1$ for $\tau'$ and $\pm 1$ for $\tau$. As the two other eigenvalues of $\tau'$ are complex conjugate numbers, the corresponding eigenvectors are complex conjugate. This implies that the eigenvalues of $\tau$ distinguished of $\Delta=\pm 1$ are equal and  from Proposition \ref{lemma1}, $\tau=\tau_4(\alpha_5,\alpha_6)$. If we assume that $k_2 \geq 2$, there exist $\tau$ and $\tau''$ not equal and belonging to $\Z_2^{k_2}$. Thus we have $\tau=\tau_4(\alpha_5,\alpha_6)$ and $\tau''=\tau_4(\alpha'_5,\alpha'_6)$. But
\[
\tau_4(\alpha_5,\alpha_6)\circ \tau_4(\alpha'_5,\alpha'_6)=\tau_4(\alpha'_5,\alpha'_6)\circ \tau_4(\alpha_5,\alpha_6) \quad \Leftrightarrow \quad \alpha_5=\alpha'_5, \quad \alpha_6=\alpha'_6
\]
and $\tau=\tau''$, this contradicts  the hypothesis.
\end{proof}

From this lemma, we have to determine, in a first step, the subgroups $\Gamma $ of $Aut (\mathfrak h_3)$ isomorphic a $(\Z_2)^k $ with $k \geq 2$.

\noindent $\bullet$ Any involutive automorphism $\tau$ commuting with $\tau_1 (\alpha_3, \alpha_6) $ with $\tau \neq  \tau_1 (\alpha_3, \alpha_6) $ is equal to
$
\tau_2(-\alpha_3,\alpha_5)$ or $ \tau_4(\alpha_5,-\alpha_6)$
and  we have
 \[
\tau_1(\alpha_3,\alpha_6) \circ \tau_2(-\alpha_3,\alpha_5) = \tau_{4}\left(-\frac{\alpha_{3}\alpha_{6}}{2}-\alpha_{5},-\alpha_{6}\right) \ {\rm and }
\]

\[
\left[\tau_2(-\alpha_3,\alpha_5),\tau_{4}\left(-\frac{\alpha_{3}\alpha_{6}}{2}-\alpha_{5},-\alpha_{6}\right)\right] = 0.
\]
Thus $ \Gamma_{7}(\alpha_3,\alpha_5,\alpha_6) = \left\{Id,\tau_1(\alpha_3,\alpha_6),\tau_2(-\alpha_3,\alpha_5),\tau_{4}
\left(-\frac{\alpha_{3}\alpha_{6}}{2}-\alpha_{5},-\alpha_{6}\right)\right\}
$ is a subgroup of $Aut (\mathfrak h_ {3})$ isomorphic to $\Z^{2}_{2} $. Moreover it is the only subgroup of $ Aut (\mathfrak h_3) $ of type $ (\Z_2)^k $, $ k \geq 2 $, containing an automorphism of type $\tau_1(\alpha_3,\alpha_6)$.

\noindent $\bullet$ A direct computation shows that any abelian subgroup $\Gamma $ containing $\tau_2 (\alpha_3 \alpha_5)$ is either isomorphic to $\Z_2 $ or  equal to $\Gamma_7.$

\noindent $\bullet$ Assume that $ \tau_3(\alpha_{1},\alpha_3,\alpha_6) \in \Gamma$. The automorphisms $\tau_3(-\alpha_{1},-\alpha_2,\alpha'_6)$ and $\tau_4(\alpha_5,\alpha'_6)$ commute with $ \tau_3(\alpha_{1},\alpha_3,\alpha_6)$.
Since 
\[
\tau_3(\alpha_{1},\alpha_2,\alpha_6) \circ \tau_3(-\alpha_{1},-\alpha_2,\alpha'_6) =\tau_4\left(\frac{\alpha'_6(1 - \alpha_{1})-\alpha_{6}(1 + \alpha_{1})}{\alpha_{2}},-\alpha_{6}-\alpha'_6\right)
\]
we obtain the following subgroup, denoted $\Gamma_8(\alpha_1,\alpha_2,\alpha_6,\alpha'_6)$:
\[
\left\{Id,\tau_3(\alpha_{1},\alpha_2,\alpha_6),\tau_3(-\alpha_{1},-\alpha_2,\alpha'_6),
\tau_4\left(\frac{\alpha'_6(1 - \alpha_{1})-\alpha_{6}(1 + \alpha_{1})}{\alpha_{2}},-\alpha_{6}-\alpha'_6\right)\right\}
\]
which is isomorphic to  $\Z^{2}_{2}$.

\noindent $\bullet$ We suppose that $ \tau_4(\alpha_5,\alpha_6) \in \Gamma$. If $ \Gamma $ is not isomorphic to $ \Z_{2} $,  then $\Gamma$ is one of the groups $\Gamma_{7}, \Gamma_{8}$.
\begin{theorem}
Any finite abelian subgroup $\Gamma$ of  $Aut (\mathfrak h_3)$ isomorphic to $(\Z_2)^k$ is one of the following
\begin{enumerate}
\item $k=1$, $\Gamma=\Gamma_1(\alpha_3,\alpha_6), \
\Gamma_2(\alpha_3,\alpha_5), \
\Gamma_3(\alpha_1,\alpha_2,\alpha_6),  \ \alpha_2 \neq 0, \
\Gamma_4(\alpha_5,\alpha_6),$
\item $k=2$, $\Gamma=\Gamma_{7}(\alpha_3,\alpha_5,\alpha_6), \ \Gamma_8(\alpha_1,\alpha_2,\alpha_6,\alpha'_6).$
    \end{enumerate}
\end{theorem}

\medskip

Assume now that $\Gamma$ is isomorphic to $\Z_3^{k_3}$ with $k_3 \geq 2$. If $\tau \in \Gamma_5$, its matricial representation is
\[
\begin{pmatrix}
\displaystyle\frac{-1 - \sqrt{-3 - 4 \alpha_{2} \alpha_{3}}}{2}&\alpha_2&0\\
\alpha_3&\displaystyle\frac{-1 + \displaystyle\sqrt{-3 - 4 \alpha_{2} \alpha_{3}}}{2}&0\\
\alpha_5&\alpha_6&1
\end{pmatrix}.
\]
To simplify, we put $\lambda=\displaystyle\frac{-1 - \sqrt{-3 - 4 \alpha_{2} \alpha_{3}}}{2}$.  The eigenvalues of $\tau$ are $1,j,j^2$ and the corresponding eigenvectors $X_3,V,\overline{V}$ with
\[
V=\displaystyle \left(1,-\displaystyle\frac{\lambda-j}{\alpha_2},-\displaystyle\frac{\alpha_5}{1-j}+
\frac{\alpha_6(\lambda-j)}{\alpha_2(1-j)}\right)
\]
if $\alpha_2 \neq 0$. If $\tau'$ is an automorphism of order $3$ commuting with $\tau$, then
\[
\tau' V=jV \qquad {\mbox or} \qquad  j^2V.
\]
But the two first components of $\tau'(V)$ are
\[
 \lambda' - \displaystyle\frac{\beta_2}{\alpha_2}(\lambda-j),\qquad  \beta_3-\frac{\lambda'(\lambda-j)}{\alpha_2}
\]
where $\beta_i$ and $\lambda'$ are the corresponding coefficients of the matrix of $\tau'$. This implies
\[
\alpha_2\lambda'-\beta_2(\lambda-j)=\alpha_2j \qquad {\mbox{\rm or}} \qquad \alpha_2j^2.
\]
Considering the real and complex parts of this equation, we obtain
\[
\left\{
\begin{array}{l}
\alpha_2\lambda'-\beta_2\lambda=0, \\
\beta_2j=\alpha_2j \qquad \mbox{or} \qquad \alpha_2j^2.
\end{array}
\right.
\]
As $\alpha_2 \neq 0$, we obtain $\alpha_2=\beta_2$ and $\lambda=\lambda'$. Let us compare the second component of $\tau'(V)$. We obtain
\[
\beta_3\alpha_2-\lambda'(\lambda-j)=-(\lambda-j)j \qquad \mbox{or} \quad  -(\lambda-j)j^2.
\]
As $\lambda=\lambda'$, we have in the first case $2\lambda j= j^2$ and in the second case $2\lambda j= j^3=1.$ In any case, this is impossible. Thus $\alpha_2=0$ and, from Section 2.2, $\tau = Id$. This implies that $k_3=1$ or $0$.
\begin{theorem}
Let $\Gamma$ be a finite abelian subgroup of $Aut(\h_3)$. Thus $\Gamma$ is isomorphic to one of the following group
\begin{enumerate}
\item $\Z_2 \times \Z_2$,
\item $\Z_2^{k_2} \times \Z_3^{k_3} \times \cdots \times \Z_p^{k_p}$ with $k_i=0 \ \mbox{or} \  1$ for $i=2, \cdots, p$.
    \end{enumerate}
    \end{theorem}
To prove the second part, we show as in the case $i=3$ that $k_i=1$ as soon as $k_i \neq 0$.

\medskip

\noindent{\bf Remark.}
We have determined the finite abelian subgroups of $Aut(\mathfrak h_3)$. There are non-abelian finite subgroups  with elements of order at most $3$. Take for example the subgroup generated by
\[
\sigma_{1}=\begin{pmatrix}
\-1&0&0\\
0&1&0\\
0&0&-1
\end{pmatrix},
\qquad
\sigma_{2} = \begin{pmatrix}
-\frac{1}{2}&\alpha&0\\
-\frac{3}{4\alpha}&-\frac{1}{2}&0\\
0&0&1
\end{pmatrix}  \qquad \alpha \neq 0.
\]
The relations on the generators are $\sigma^{2}_{1} = Id,\  \sigma^{3}_{2} = Id, \ \sigma_{1}\sigma_{2}\sigma_{1} = \sigma^{2}_{2}.
$
Thus the group generated by $\sigma_1$ and $\sigma_2$ is isomorphic to the symmetric group $\Sigma_3$ of degree  
$3$.

\subsection{Description of the $\Z_2$ and $\Z_2^2$-gradings of $\h_3$}
Let $\Gamma$ be a finite abelian subgroup of $Aut(\h_3)$ isomorphic to $\Z_2^k$ ($k=1$ or $2$). 

$\bullet$ If $\Gamma = \Z_2$, we have obtained $\Gamma=\Gamma_i$, $i=1,2,3,4.$ Up to equivalence of gradings, the $\Z_2$-grading of $\h_3$ are:
$$\h_3=\R\{X_2\} \bigoplus \R\{X_1,X_3\} \quad {\rm and} \quad \h_3=\R\{X_1\} \bigoplus \R\{X_2,X_3\}.$$

\medskip

$\bullet$ If  $\Gamma = \Z^{2}_2$ then  $\Gamma = \Gamma_7$ or  $\Gamma = \Gamma_8$. 
\begin{lemma}
There is an automorphism $\sigma \in Aut(\mathfrak h_{3})$ such that
\[
\sigma^{-1} \Gamma_{7} \sigma = \Gamma_{8}.
\]
\end{lemma}
\noindent The proof is a simple computation.
There also exists $\sigma\in Aut(\mathfrak h_3)$ such that
\[
\left\{
\begin{array}{l}
 \sigma^{-1} \tau_{1}(\alpha_{3},\alpha_{6}) \sigma = \tau_{1}(0,0), \\
 \sigma^{-1} \tau_{2}(-\alpha_{3},\alpha_{5}) \sigma = \tau_{2}(0,0).
 \end {array}
 \right.
 \]
We deduce:
\begin{proposition}
Every $\Z^{2}_{2}$-grading on $\mathfrak h_{3}$ is equivalent to the grading defined by
\[
\Gamma_{7}(0,0,0) = \{ Id,\tau_{1}(0,0),\tau_{2}(0,0),\tau_{4}(0,0)\}.
\]
\end{proposition}
\noindent This grading corresponds to
\[
\mathfrak h_{3} =\{0\} \oplus \R(X_{1})\oplus \R(X_{2}) \oplus \R(X_{3}).
\]

\subsection{Non existence of Riemannian symmetric structures on $\mathbb{H}_3/H$}

Consider the symmetric space 
$\mathbb{H}_3/H_1$ associated with the grading 
$$\h_3=\R \{X_2\} \bigoplus \R \{X_1, X_3\}.$$
 Let $\{ \omega_1, \omega_2 ,\omega_3 \}$ be the dual basis of $\{X_1,X_2,X_3 \}$. Any pseudo Riemannian metric on the symmetric space $\mathbb{H}_3/H_1$ where $H_1$ is a one-dimensional connected Lie group whose Lie algebra $\g_0=\R (X_2)$ is given by a non degenerate bilinear form $B=a\omega_1 ^2+ b \omega_1 \wedge \omega_3 + c \omega_3 ^2 $ on $\g_1=\R (X_1,X_3)$ which is $ad \g_0$-invariant. This implies
$$B([X_2,X_1],X_3)=-B(X_3,X_3)=-c=0.$$
But we have also
$$B([X_2,X_1],X_1)+B(X_1,[X_2,X_1])=-2B(X_3,X_1)=-2b=0.$$
We deduce
\begin{proposition}
The nilpotent symmetric space $\mathbb{H}_3/H$ associated to the grading
 $$\h_3=\R \{X_2\} \bigoplus \R \{X_1, X_3\}$$
 doesn't admit any pseudo-Riemannian symmetric metric.
 \end{proposition}

\medskip

Consider now the symmetric space $\mathbb{H}_3/H_2$ associated with the grading 
$$\h_3=\R \{X_3\} \bigoplus \R \{X_1, X_2\}.$$
Then $H_2$ is the Lie subgroup whose Lie algebra is $\R\{X_3\}$ and the bilinear form $B=a\omega_1 ^2+ b \omega_1 \wedge \omega_2 + c \omega_2 ^2 $ on $\g_1=\R (X_1,X_2)$ is $ad X_3$-invariant because $ad X_3=0$.  But $Ad_G$ is an homomorphism of $G$ onto the group of inner automorphisms of $\g$ with kernel the center of $G$, we deduce that $Ad_G(H)$ is compact in this case and any non degenerate bilinear form $B$ on $\g_1$ defines a Riemannian or a Lorentzian structure on the symmetric space $\mathbb{H}_3/H_2.$
\begin{proposition}
The nilpotent symmetric space $\mathbb{H}_3/H_2$ associated to the grading
 $$\h_3=\R \{X_3\} \bigoplus \R \{X_1, X_2\}$$
 admits a structure of Riemannian symmetric space. It admits also a structure of Lorentzian symmetric space.
 \end{proposition}

\subsection{Riemannian $\Z_2^2$-symmetric structures on $\mathbb{H}_3$}

Consider on $\mathbb{H}_3$ a  $\Z^2_2 $-symmetric structure. It is determined, up to equivalence, by the $\Z^2_2 $-grading of $\mathfrak h_3$
\[
\mathfrak h_{3} =\{0\} \oplus \R(X_{1})\oplus \R(X_{2}) \oplus \R(X_{3}).
\]
Since every automorphism of $\mathfrak h_3 $ is an isometry of any invariant  Riemannian metric on
$\mathbb H
_3 $, we deduce
\begin{theorem}
Any Riemannian structure $\Z^2_2 $-symmetric over $\mathbb H
_3 $ is isometric to the Riemannian structure associated with the grading
\[
\mathfrak h_{3} =\{0\} \oplus \R(X_{1})\oplus \R(X_{2}) \oplus \R(X_{3})
\]
and the $\Z_2^2$-symmetric Riemannian metric is written
\[
g = \omega^{2}_{1} + \omega^{2}_{2} + \lambda^{2} \omega^{2}_{3} \qquad {\rm with} \qquad \lambda \neq 0,
\]
 where $\{\omega_1,\omega_2,\omega_3\}$ is the dual basis of $\{X_1,X_2,X_3\}$.
\end{theorem}
\begin{proof}
Indeed, since the components of the grading are orthogonal, the Riemannian metric $ g $, which coincides with the form $ B $ satisfies
\[
g=\alpha_{1} \omega^{2}_{1} + \alpha_{2} \omega^{2}_{2} + \alpha_{3} \omega^{2}_{3} \qquad {\rm with} \qquad  \alpha_{1} > 0,\alpha_{2}> 0,\alpha_{3}> 0.
\]
According to \cite{[G.P1]}, we reduce the coefficients to  $\alpha_1 =\alpha_2 =1$.
\end{proof}

\medskip

\noindent{\bf Remark.} According to \cite{[H]} and \cite{[G.P2]}, this metric is naturally reductive for any $\lambda$.

\begin{corollary} A Riemannian tensor $g$ on $\mathbb H_3 $ determines a $\Z^2_2 $-symmetric Riemannian structure  over $\mathbb H
_3 $ if and only if it is a left-invariant metric on $\mathbb H_3 $.
\end{corollary}
This is a consequence of the previous theorem and of the classification of left-invariant metrics on Heisenberg groups (\cite{[G.P1]}).

\subsection{Lorentzian  $\Z^{2}_{2} $-symmetric structures on $\mathbb H_{3}$}
We say that an homogeneous space $(M = G/H, g)$ is {\it Lorentzian} if the canonical action of $G$ on $M$ preserves a Lorentzian metric (i.e. a smooth field  of non degenerate quadratic forms of signature $(n-1,1)$) (see \cite{[C]}).
\begin{proposition}[\cite{[D.Z]}]
Modulo an automorphism and a multiplicative constant, there exists on $\mathbb H
_{3}$ one left-invariant metric assigning a strictly positive length on the center of
$\mathfrak h_{3} $.
\end{proposition}
The Lie algebra $\h_ {3} $ is generated by the central vector $X_3$ and  $X_1$ and $X_2 $ such that $ [X_1, X_2] =  X_3$. The automorphisms of the Lie algebra preserve the center and then send the element $X_3$ on $\lambda X_3$, with $\lambda \in \R ^{*}$. Such an automorphism acts on the plane generated by $X_1$ and $X_2 $ as an automorphism of determinant $\lambda $. \\
It is shown in \cite {[R]} and \cite {[R.R]} that, modulo an automorphism of $\mathfrak h_ {3} $, there are three classes of invariant Lorentzian metrics on $\mathbb H
_{3}$,  corresponding to the cases where $||X_3||$ is negative, positive or zero.\\
We propose to look at the Lorentzian metrics that are associated with  the  $\Z_2^2$-symmetric structures over $\mathbb H
_3$.
If  $\mathfrak g$  is the Heisenberg algebra equipped with a $\Z^2_2 $-grading, then by automorphism, we can  reduce to the case where
$\Gamma = \Gamma_7$. In this case, the grading of $\mathfrak  h_{3} $ is given by:
\[
\mathfrak h_{3}= \mathfrak g_{0} + \mathfrak g_{+ -} + \mathfrak g_{- +} +\mathfrak g_{- -}
\]
with $\mathfrak g_{0} = \{0\},$ and
\[  \mathfrak g_{+ -} =  \R\left(X_{2} -\frac{\alpha_{6}}{2}X_{3}\right), \quad
\mathfrak g_{- +} =  \R\left(X_{1} -\frac{\alpha_{3}}{2}X_{2}+ \frac{\alpha_{5}}{2}X_{3}\right),\quad
\mathfrak g_{- -}  = \R\left(X_{3}\right).
\]
Assume $
Y_{1} = X_{1} -\frac{\alpha_{3}}{2}X_{2}+ \frac{\alpha_{5}}{2}X_{3}, \; Y_{2} = X_{2} -\frac{\alpha_{6}}{2}X_{3}, \; Y_{3} = X_{3}.
$
The dual basis is
\[
\vartheta_{1} = \omega_{1} \qquad \vartheta_{2} = \omega_{2} + \frac{\alpha_{3}}{2} \omega_{1} \qquad \vartheta_{3} = \omega_{3} -\frac{\alpha_{6}}{2}\omega_{2} -\left(\frac{\alpha_{3}\alpha_{6}}{4} + \frac{\alpha_{5}}{2}\right)\omega_{1}
\]
where $\{\omega_1,\omega_2,\omega_3\}$ is the dual basis of the base $\{X_1,X_2,X_3\}$.\\
\medskip

\noindent {\bf Case $I$} The components $ \mathfrak g_{+ -},\, \mathfrak g_{- +}, \,\mathfrak g_{- -}$ are non-degenerate. The quadratic form induced on $\mathfrak  h_3$ therefore writes
\[
g = \lambda_{1}\omega^{2}_{1} +
\lambda_{2}\left(\omega_{2} + \frac{\alpha_{3}}{2} \omega_{1}\right)^{2} +
\lambda_{3} \left(\omega_{3} - \frac{\alpha_{6}}{2} \omega_{2} -
\left(\frac{\alpha_{5}}{2} + \frac{\alpha_{3} \alpha_{6}}{4}\right)\omega_{1}\right)^{2}
\]
with $\lambda_{1},\lambda_{2},\lambda_{3} \neq 0$.
The change of basis associated with the matrix
\[
\begin{pmatrix}
1&0&\\
\frac{\alpha_{3}}{2}&1&0\\
-\frac{\alpha_{5}}{2} - \frac{\alpha_{3} \alpha_{6}}{4}&- \frac{\alpha_{6}}{2}&1
\end{pmatrix}
\]
is an automorphism. Thus $ g $ is isometric to
\[
g = \lambda_{1}\omega^{2}_{1} +
\lambda_{2}\omega^{2}_{2} + \lambda_{3}\omega^{2}_{3} .
\]
Since the signature is $(2,1) $ one of the $\lambda_ i$ is negative and the two others positive.

\begin{proposition}
Every Lorentzian metric $\Z^2_2 $-symmetric $ g $ on $\mathbb H
_3$ such that the components of the grading of $\mathfrak {h}_3 $ are non degenerate, is
reduced to one of these two forms:
$$
g = - \omega^{2}_1 + \omega^{2}_2 + \lambda^{2} \omega^{2}_3 \qquad {\rm or} \qquad
g = \omega^{2}_1 + \omega^{2}_2  - \lambda^{2} \omega^{2}_3
$$

\end{proposition}
\medskip
\noindent {\bf Case $II$} Suppose that a component is degenerate.
When this component is $\R (X_2 +\frac {\alpha_6}{2} X_3) $ or $\R (X_1 -\frac {\alpha_3} {2} X_2 +\frac{\alpha_5}{2}  X_3 ) $ then, by automorphism, it reduces to the above case.\\
Suppose then that the component containing the center is degenerate.\\
Thus the quadratic form induced on $\mathfrak h_3$ is written
$$g = \omega^{2}_{1} +
\left[\omega_{3} - \frac{\alpha_{6}}{2} \omega_{2} -
\frac{2\alpha_{5} + \alpha_{3} \alpha_{6}}{4} \omega_{1}\right]^{2}
- \left[\omega_{2} - \omega_{3} + \frac{\alpha_{6}}{2} \omega_{2} +
\frac{2\alpha_{5}+ \alpha_{3} \alpha_{6}}{4}\omega_{1}\right]^{2}.$$

The change of basis associated with the matrix
\[
\begin{pmatrix}
1&0&\\
\frac{\alpha_{3}}{2}&1&0\\
-\frac{\alpha_{5}}{2} - \frac{\alpha_{3} \alpha_{6}}{4}&- \frac{\alpha_{6}}{2}&1
\end{pmatrix}
\]
is given by an automorphism. Thus $g$ is isomorphic to
\[
g = \omega^{2}_{1} +  \omega^{2}_{3} -
(\omega_{2} - \omega_{3})^{2}.
\]
\begin{proposition}
Every Lorentzian  $\Z^2_2 $-symmetric $ g $ metric on $\mathbb H_3$ such that the component of the grading of $\mathfrak {h} _3 $ containing the center is degenerate, is reduced to the form
\[ g = \omega^{2}_{1} +  \omega^{2}_{3} -
(\omega_{2} - \omega_{3})^{2}.
\]
\end{proposition}

\begin{corollary} A Lorentzian tensor $g$ on $\mathbb H_3 $ determines a $\Z^2_2 $-symmetric Lorentzian structure 
 over $\mathbb H
_3 $ if and only if it is a left-invariant Lorentzian metric on $\mathbb H_3 $.
\end{corollary}
The classification, up to isometry,  of left-invariant Lorentzian metrics on $\HH_3$ is described in \cite{[C.P]} and in \cite{[R.R]}.  It corresponds to the previous classification of Lorentzian $\Z_2^2$-symmetric metrics. 
\section{$\Z_2^k$-symmetric spaces based on $\mathbb{H}_{2p+1}$}

\subsection{$\Z_2^k$-gradings of $\h_{2p+1}$}
Let $\sigma$ be an involutive automorphism of the $(2p+1)$-dimensional Heisenberg algebra $\h_{2p+1}$. Let $\{X_1,\cdots,X_{2p+1}\}$ be a basis of $\h_{2p+1}$ whose structure constants are given by
$$[X_1,X_2]=\cdots=[X_{2p-1},X_{2p}]=X_{2p+1}.$$
Since the center $\R\{X_{2p+1}\}$ is invariant by $\sigma$, it is contained in an homogeneous component of the grading $\h_{2p+1}=\g_0 \oplus \g_1$ associated with $\sigma$. But for any $X \in \h_{2p+1}$, $X \neq 0$, there exists $Y \neq 0$ such that $[X,Y]=aX_{2p+1}$ with $a \neq 0$. We deduce that any $\Z_2$-grading is equivalent to one of the following:
\begin{enumerate}
  \item If $X_{2p+1} \in \g_0$, then
  \begin{itemize}
   \item $\h_{2p+1}=\R\{X_{2p+1}\} \oplus \R\{X_1,X_2,\cdots,X_{2p}\}$
   \item $\h_{2p+1}=\R\{X_1,X_2,X_3,\cdots,X_{2k},X_{2p+1}\} \oplus \R\{X_{2k+1},X_{2k+2},\cdots,X_{2p}\}$
                                            \end{itemize}
  \item If $X_{2p+1} \in \g_1$, then
\begin{itemize}
  \item $\h_{2p+1}=\R\{X_1,X_3,\cdots,X_{2p-1}\} \oplus \R\{X_{2},X_{4},\cdots,X_{2p},X_{2p+1}\}$
  \item $\h_{2p+1}=\R\{X_2,X_4,\cdots,X_{2p}\} \oplus \R\{X_{1},X_{3},\cdots,X_{2p-1},X_{2p+1}\}.$
\end{itemize}
\end{enumerate}
Let $\h_{2p+1}=\bigoplus_{\gamma \in \Z_2^k}\g_\gamma$ be a $\Z_2^k$-grading of the Heisenberg algebra. The support of this grading is the subset $\{\gamma \in \Z_2^k, \ \g_\gamma \neq 0\}.$ We will say that this grading is irreducible if the subgroup of $\Z_2^k$ generates by its support is the full group $\Z_2^k$.
\begin{lemma}
If $\h_{2p+1}$ admits an irreducible $\Z_2^k$-grading, then $k=1$ or $k=2$.
\end{lemma}
In fact, this is a consequence of the previous classification of the $\Z_2$-gradings of $\h_{2p+1}$. We deduce also that any $\Z_2^2$-grading is equivalent to
 $$\h_{2p+1}=\{0\}\oplus \R\{X_{2p+1}\} \oplus \R\{X_1,X_3,\cdots,X_{2p-1}\}\oplus \R\{X_2,X_4,\cdots,X_{2p}\}.$$

\subsection{Pseudo-Riemannian symmetric spaces  $\mathbb{H}_{2p+1}/H$}

We consider the symmetric spa\-ces $\mathbb{H}_{2p+1}/H$  corresponding to the previous symmetric decomposition of $\h_{2p+1}$, where $H$ is a connected Lie subgroup of $\mathbb{H}_{2p+1}$ whose Lie algebra is $\g_0$.

\noindent$\bullet$ With the $\Z_2$-grading $\h_{2p+1}=\R\{X_{2p+1}\} \oplus \R\{X_1,X_2,\cdots,X_{2p}\}.$
Since $ad(X_{2p+1})$ is zero

\noindent  any non degenerate bilinear form on $\g_1$ defines a symmetric pseudo-riemannian metric on $\mathbb{H}_{2p+1}/H$ where $H$ is a connected one dimensional Lie Group.

\noindent$\bullet$ Consider the $\Z_2$-grading
$$\h_{2p+1}=\R\{X_1,X_2,X_3,\cdots,X_{2k},X_{2p+1}\} \oplus \R\{X_{2k+1},X_{2k+2},\cdots,X_{2p}\}$$
In this case, $H$ is a Lie subgroup isomorphic to $\mathbb{H}_{2k+1}.$ 
Since we have $[\g_0,\g_1]=0,$
any non degenerate bilinear form on $\g_1$ defines a symmetric pseudo-Riemannian metric on $\mathbb{H}_{2p+1}/\mathbb{H}_{2k+1}.$

\noindent$\bullet$ We consider the $\Z_2$-gradings
$$\begin{array}{ll}
& \h_{2p+1}=\R\{X_1,X_3,\cdots,X_{2p-1}\} \oplus \R\{X_{2},X_{4},\cdots,X_{2p},X_{2p+1}\} \\
{\rm or} & \h_{2p+1}=\R\{X_2,X_4,\cdots,X_{2p}\} \oplus \R\{X_{1},X_{3},\cdots,X_{2p-1},X_{2p+1}\}.
\end{array}$$
In this case, any bilinear form on $\g_1$ which is $ad(\g_0)$-invariant is degenerate. In fact, if $B$ is such a form, we have
$$B([X_{2k+1},X_{2k+2}],X_1)=B(X_{2p+1},X_{2p+1})=0$$
and for any $k=0,\cdots,p-1$ and $s \neq k+1$
$$B([X_{2k+1},X_{2k+2}],X_{2s})=B(X_{2p+1},X_{2s})=0, $$
and $X_{2p+1}$ is in the kernel of $B$. We have the same proof for the second grading.
\begin{proposition}
The symmetric spaces $\mathbb{H}_{2p+1}/H$ corresponding to the $\Z_2$-grading of $\h_{2p+1}$:
\begin{itemize}
  \item $\h_{2p+1}=\R\{X_1,X_3,\cdots,X_{2p-1}\} \oplus \R\{X_{2},X_{4},\cdots,X_{2p},X_{2p+1}\}$
  \item $\h_{2p+1}=\R\{X_2,X_4,\cdots,X_{2p}\} \oplus \R\{X_{1},X_{3},\cdots,X_{2p-1},X_{2p+1}\}$
\end{itemize}
are not pseudo-Riemannian symmetric spaces.
\end{proposition}
\subsection{Pseudo-Riemannian $\Z_2^2$-symmetric spaces  $\mathbb{H}_{2p+1}/H$}
Let us consider the 
\noindent $\Z_2^2$-gra\-ding of the Heisenberg algebra
$$\h_{2p+1}=\{0\}\oplus \R\{X_{2p+1}\} \oplus \R\{X_1,X_3,\cdots,X_{2p-1}\}\oplus \R\{X_2,X_4,\cdots,X_{2p}\}.$$
Since $\g_0=\{0\}$, then $H$ is reduced to the identity and the $\Z_2^2$-symmetric spaces  $\mathbb{H}_{2p+1}/H$ is isomorphic to $\mathbb{H}_{2p+1}$. The reductive decomposition $\h_{2p+1}=\g_0 \oplus \m$ is reduced to $\m$. Since $\g_0=\{0\}$, any bilinear definite positive form on $\m$ for which the homogeneous components $\R\{X_{2p+1}\}$, $\R\{X_1,X_3,\cdots,X_{2p-1}\}$ and $ \R\{X_1,X_4,\cdots,X_{2p}\}$ are pairwise orthogonal defines a $\Z_2^2$-symmetric Riemannian structure on  $\mathbb{H}_{2p+1}$.

The Levi-Civita connection associated with this Riemannian metric is an affine connection, that is, it is torsion-free and the curvature tensor $R$ satisfies $\nabla R=0$, where $\nabla$ is the covariant derivative of this connection. In case of Riemannian symmetric space, the Levi-Civita connection associated with the symmetric Riemannian metric corresponds to the canonical connection defined in \cite{Kobayashi-Nomizu-2} which defines the natural affine structure on a symmetric space. This is not the case for Riemannian-$\Z_2^2$-symmetric spaces. In the next section, we define a class of affine connection adapted to the  $\Z_2^2$-symmetric structures, and we prove, in case of the Riemannian $\Z_2^2$-symmetric space  $\mathbb{H}_{2p+1}/H$, that there exist adapted connection with torsion and curvature-free.

\subsection{Adapted affine connections on the $\Z_2^2$-symmetric spaces  $\mathbb{H}_{2p+1}/H$}

Let $G/H$ be a $\Z_2^k$-symmetric space. Since $G/H$ is a reductive homogeneous space, that is $\g$ admits a decomposition $\g=\g_0+\m$ with $[\g_0,\g_0] \subset \g_0$ and $[\g_0,\m] \subset \m$, any connection is given by a linear map
 $$\bigwedge : \m \rightarrow gl(\m)$$
 satisfying
 $$\displaystyle \bigwedge\left[X,Y\right]=\left[\bigwedge (X), \lambda (Y)\right]$$
 for all $X \in \m$ and $Y \in \g_0$, where $\lambda$ is the linear isotropy representation of $\g_0$. The corresponding torsion and curvature tensors are given by:
 $$\begin{array}{ll}
& T(X,Y)=\bigwedge(X)(Y) - \bigwedge(Y)(X)-[X,Y]_{\m} \\
 {\rm and} &
\displaystyle  R(X,Y)=\left[\bigwedge(X),\bigwedge(Y)\right]-\bigwedge[X,Y]-\lambda([X,Y]_{\g_0})
 \end{array}$$
 for any $X,Y \in \m.$
 \begin{definition}
Consider the affine connection on the $\Z_2^k$-symmetric space $G/H$ defined by the linear map
$$\bigwedge : \m
 \rightarrow gl(\m).$$
Then this connection is called adapted to the $\Z_2^k$-symmetric structure, if any
$$\bigwedge(X_\gamma)(\g_{\gamma'}) \subset \g_{\gamma\gamma'}$$
for any $\gamma,\gamma' \in \Z_2^k$, $\gamma,\gamma \neq 0.$ The connection is called homogeneous if any
homogeneous component $\g_{\gamma}$ of $\m$ is invariant by $\bigwedge$.
\end{definition}
Now we consider the case where $G/H=\mathbb{H}_{2p+1}/H$ is the $\Z^2_2$-symmetric space defined by the grading
$$\h_{2p+1}=\{0\}\oplus \R\{X_{2p+1}\} \oplus \R\{X_1,X_3,\cdots,X_{2p-1}\}\oplus \R\{X_2,X_4,\cdots,X_{2p}\}.$$
We have seen that $H$ is reduced to the identity and $\mathbb{H}_{2p+1}/H$ is isomorphic to $\mathbb{H}_{2p+1}.$  Consider an adapted connection and let $\bigwedge$ be the associated linear map. Since the connection is adapted to the $\Z_2^2$-symmetric structure, $\bigwedge$ satisfies:
$$
\left\{
\begin{array}{lll}
\medskip
\bigwedge(X_{2k+1})(X_{2l+1})=\bigwedge(X_{2s})(X_{2t})=0, \quad &k,l=0,\cdots,p-1,  \quad &s,t=1,\cdots,p,\\
\medskip
\bigwedge(X_{2k+1})(X_{2s})=C_s^{2k+1}X_{2p+1}, & s=1,\cdots,p, & k=0, \cdots,p-1,\\
\medskip
\bigwedge(X_{2s})(X_{2k+1})=C_k^{2s}X_{2p+1}, &s=1,\cdots,p, & k=0, \cdots,p-1,\\
\medskip
\bigwedge(X_{2k+1})(X_{2p+1})=\sum_{s=1}^p a_{2k+1}^sX_{2s}, & k=0, \cdots,p-1, &\\
\medskip
\bigwedge(X_{2s})(X_{2p+1})=\sum_{k=0}^p a_{2s}^kX_{2k+1}, & s=1, \cdots,p. &\\
\end{array}
\right.
$$
\begin{theorem}
Any adapted connection $\nabla$ on the $\Z_2^2$-symmetric space $\mathbb{H}_{2p+1}/H=\mathbb{H}_{2p+1}$ satisfies $T=0$ and $ R=0$
where $T$ and $R$ are respectively the torsion and the curvature of $\nabla$ if and only if the corresponding linear map $\bigwedge$ satisfies
$$\begin{array}{l}
\left\{
\begin{array}{ll}
\medskip
\bigwedge(X_{2k+1})(X_{2s})=C_s^{2k+1}X_{2p+1}, \quad & s=1,\cdots,p, \quad k=0, \cdots,p-1,\\
\medskip
\bigwedge(X_{2k+1})(X_{i})=0, & k=0, \cdots,p-1, \quad i \notin \{2,\cdots,2p\},
\medskip
\end{array}
\right.
\\ \medskip
\left\{
\begin{array}{ll}
\medskip
\bigwedge(X_{2s})(X_{2k+1 })=C_s^{2k+1}X_{2p+1}, \quad & s=1,\cdots,p, \qquad k=0, \cdots, p-1, k \neq s-1, \\ \medskip
\bigwedge(X_{2s})(X_{2s-1 })=(C_s^{2k+1}-1)X_{2p+1}, & s=1, \cdots,p,\\
\medskip
\bigwedge(X_{2s})(X_{i})=0, & s=1, \cdots,p, \qquad i \notin \{1,\cdots,2p-1\}.
\end{array}
\right.
\end{array}
$$
\end{theorem}
In fact, we determine in a first step, all the connection adapted to the  $\Z_2^2$-symmetric structure and which are torsion-free. In this case, $\bigwedge$ satisfies
$$\bigwedge(X)(Y)-\bigwedge(Y)(X)-[X,Y]=0, \ {\rm for \ any } \ X,Y \in \h_{2p+1}.$$ 


\begin{thebibliography}{}

\bibitem{[B.G]}  Bahturin, Y., Goze, M.;  \emph{$\Z_{2}\times \Z_{2}$-symmetric spaces}. Pacific J. Math. 236, no. 1, 1-21, 2008.
\bibitem{[C]}  Calvaruso, G.; \emph{Homogeneous structures on three-dimensional Lorentzian manifolds}. J. Geom. Phys. 57, no. 4, 1279 - 1291, 2007.
\bibitem{[C.P]} Cordero, L.A., Parker, P.E.; \emph{Left-invariant Lorentzian metrics on $3$-dimensional Lie groups}. Rend. Mat. Appl. (7) 17, no. 1, 129 -155, 1997.
\bibitem{[D.Z]} Dumitrescu,  S., Zeghib, A.; \emph{G\'eom\'etrie Lorentziennes de dimension $3$: classification et compl\'etude}, Geom. Dedicata (2010) 149, 243 - 273.
\bibitem{[G.P1]}  Goze, M., Piu, P.; \emph{ Classification des m\'etriques invariantes \`a gauche sur le groupe de Heisenberg}, Rend. Circ. Mat. Palermo (2) 39 (1990), no. 2, 299 -306.
\bibitem{[G.P2]}  Goze, M., Piu, P.; \emph{ Une caract\'erisation riemannienne du groupe de Heisenberg}. Geom. Dedicata 50 (1994), no. 1, 27 - 36.
\bibitem{[G.R]} Goze, M., Remm, E.;  \emph{Riemannian $\Gamma$-symmetric spaces}. Differential geometry, 195 - 206, World Sci. Publ., Hackensack, NJ, 2009.
\bibitem{[G.W]} Gray, A., Wolf, J. A.; \emph{Homogeneous spaces defined by Lie group automorphisms. I}, J. Differential Geometry 2 (1968), 77-114.
\bibitem{[H]} Hangan, Th.; \emph{Au sujet des flots riemanniens sur le groupe nilpotent de Heisenberg}.   Rend. Circ. Mat. Palermo (2) 35 (1986), no. 2, 291-305.
\bibitem{Kobayashi-Nomizu-2} Kobayashi, Sh.; Nomizu, K.,
\newblock {\em Foundations of differential geometry}, volume~II.
\newblock Interscience Publishers John Wiley and Sons, Inc., New York-London-Sydney, 1969.
\bibitem{[Ko]} Kollross, A.; \emph{Exceptional $Z_2 \times Z_2$-symmetric spaces}. Pacific J. Math. 242 (2009), no. 1, 113-130.
\bibitem{[K]} Kowalski O.; \emph{Generalized symmetric spaces}. Lecture Notes in Mathematics, 805. Springer-Verlag, Berlin-New-York, 1980.
\bibitem{[L]}Lutz, R.; \emph{Sur la g\'eom\'etrie des espaces $\Gamma$-sym\'etriques}. C. R. Acad. Sci. Paris S\'er. I Math. 293 (1981), no. 1, 55-58.
\bibitem{L.O} Ledger, A.J., Obata, M.; \emph{Affine and Riemannian $s$-manifolds}. J. Differential Geometry 2 1968 451- 459.
\bibitem{Mrugala} Mrugala , R.; \emph{On a Riemannian metric on contact thermodynamic spaces}. Proceedings of the XXVIII Symposium on Mathematical Physics (Toru\'n, 1995). Rep. Math. Phys. 38 (1996), no. 3, 339-348.
\bibitem{[N]} Nomizu, K.; Left-invariant Lorentz metrics on Lie groups, Osaka J. Math. 16 (1979) 143-150.
\bibitem{[P.R]} Piu P., Remm E.;  Riemannian symmetries in flag manifolds. Arch. Math. (Brno) 48 (2012), no. 5, 387-398.
\bibitem{[R]}Rahmani, S.; \emph{M\'etriques de Lorentz sur les groupes de Lie unimodulaires, de dimension trois}, J. Geom. Phys. 9 (1992), no. 3, 295 -302.
\bibitem{[R.R]} Rahmani, N., Rahmani, S.; \emph{Lorentzian geometry of the Heisenberg group}. Geom. Dedicata 118 (2006), 133 -140.
\bibitem{[G.R prod tens]} Remm, E., Goze, M.; \emph{On algebras obtained by tensor product}. J. Algebra 327 (2011), 13-30.

\end{thebibliography}


%
\bigskip
\end{document}